\documentclass[twoside,reqno,11pt]{amsart}

\usepackage{amsthm}
\usepackage{amsmath}
\usepackage{latexsym}
\usepackage{amsfonts}
\usepackage{amssymb}

\newtheoremstyle{theorem}
  {15pt}
  {15pt}
  {\sl}
  {\parindent}
  {\sc}
  {. }
  { }
  {}
\theoremstyle{theorem}

\newtheorem{theorem}{Theorem}[section]

\newtheoremstyle{defi}
  {15pt}
  {15pt}
  {\rm}
  {\parindent}
  {\sc}
  {. }
  { }
  {}
\theoremstyle{defi}
\newtheorem{definition}{Definition}[section]

\usepackage{hyperref}

\newcommand{\R}{\mathbb{R}}
\newcommand{\N}{\mathbb{N}}
\newcommand{\opbac}{D^{\beta,\alpha}_{1-\gamma}}
\newcommand{\caprb}{{{^C_xD}_b^\beta}}
\newcommand{\capla}{{{^C_aD}_x^\alpha}}
\newcommand{\intlb}{{_aI_x^{1-\beta}}}
\newcommand{\intra}{{_xI_b^{1-\alpha}}}

\title[Fractional calculus of variations \dots]{Fractional calculus of variations\\ [4pt]
for a combined Caputo derivative}

\author[A. B. Malinowska, D. F. M. Torres]{Agnieszka B. Malinowska $^1$, Delfim F. M. Torres $^2$}

\thanks{This is a preprint of a paper whose final and definite form has been published in:
Fract. Calc. Appl. Anal., Vol.~14, No~4 (2011), pp.~523--537; DOI: 10.2478/s13540-011-0032-6}

\begin{document}

\vbox to 2.5cm { \vfill }

\bigskip \medskip


\begin{abstract}

We generalize the fractional Caputo derivative to the fractional
derivative ${{^CD}^{\alpha,\beta}_{\gamma}}$,
which is a convex combination of the left Caputo
fractional derivative of order $\alpha$ and the right Caputo
fractional derivative of order $\beta$. The fractional variational
problems under our consideration are formulated
in terms of ${{^CD}^{\alpha,\beta}_{\gamma}}$.
The Euler--Lagrange equations for the basic and isoperimetric
problems, as well as transversality conditions, are proved.

\medskip


\smallskip


\end{abstract}


\maketitle

\vspace*{-16pt}


\section{Introduction}

The history of Fractional Calculus (FC) goes back more than three
centuries, when in 1695 the derivative of order $\alpha=1/2$ was
described by Leibniz. Since then, the new theory turned out to be
very attractive to mathematicians and many different forms of
fractional operators were introduced: the Grunwald--Letnikov,
Riemann--Liouville, Riesz, and Caputo fractional derivatives
(see, \textrm{e.g.}, \cite{kilbas,Podlubny,samko}),
and the more recent notions of \cite{klimek},
\cite{Cresson}, and \cite{Jumarie4,Jumarie3b}.
Besides mathematics, fractional
derivatives and integrals appear in physics, mechanics, engineering,
elasticity, dynamics, control theory, economics, biology, chemistry,
etc. (\textrm{cf.} \cite{ReviewerAsked02,MyID:179,D:D,ReviewerAsked01}
and references therein). The FC is nowadays covered by several books
(\textrm{e.g.}, \cite{hilfer2,kilbas,miller,Oldham,Podlubny,samko})
and a large number of relevant papers (see, \textrm{e.g.},
\cite{agrawalCap,comRic:Leitmann,Baleanu4,NonlinearDyn,Jumarie4,Jumarie3b,klimek,Rabei2,Tarasov}).

The calculus of variations is an old branch of optimization theory
that has many applications both in physics and geometry
(\textrm{cf.} \cite{MalTor,MR2164553,Trout,vanBrunt} and references therein).
Apart from a few examples known since ancient times such as Queen Dido's
problem (reported in {\it The Aeneid} by Virgil), the problem of
finding optimal curves and surfaces has been posed first by
physicists such as Newton, Huygens, and Galileo. Their contemporary
mathematicians, starting with the Bernoulli brothers and Leibniz,
followed by Euler and Lagrange, invented the calculus of
variations of a functional in order to solve those problems.
Fractional Calculus of Variations (FCV) unifies the calculus of
variations and the fractional calculus, by inserting fractional
derivatives into the variational integrals. This occurs naturally in many
problems of physics or mechanics, in order to provide more accurate
models of physical phenomena (see, \textrm{e.g.},
\cite{R:A:D:10,Atanackovic,ReviewerAsked02}).
The FCV started in 1996 with the work
of \cite{rie}. Riewe formulated the problem of the calculus of
variations with fractional derivatives and obtained the respective
Euler--Lagrange equations, combining both conservative and
nonconservative cases. Nowadays the FCV is a subject under strong
research. Different definitions for fractional derivatives and
integrals are used, depending on the purpose under study.
Investigations cover problems depending on Riemann--Liouville
fractional derivatives (see, \textrm{e.g.},
\cite{Agrawal:2002,Almeida2,Atanackovic,Baleanu:Mus,El-Nabulsi:Torres:2008,Frederico:Torres1}),
the Caputo fractional derivative (see, \textrm{e.g.},
\cite{agrawalCap,Baleanu:Agrawal,MalTor2010}),
the symmetric fractional derivative (see, \textrm{e.g.}, \cite{klimek}),
the Jumarie fractional derivative (see, \textrm{e.g.},
\cite{R:A:D:10,comRic:RL-modified,Jumarie4,Jumarie3b,Odz:Tor}), and others (see, \textrm{e.g.},
\cite{CD:Agrawal:2007,Cresson,El-Nabulsi:Torres07,withGasta:SI:Leit:85}).
Although the literature of FCV is already vast, much remains to be done.

In this paper we extend the notion of the Caputo fractional
derivative to the fractional derivative ${{^CD}^{\alpha,\beta}_{\gamma}}$, which is a convex
combination of the left Caputo fractional derivative of order $\alpha$ and
the right Caputo fractional derivative of order $\beta$. This idea
goes back at least as far as \cite{klimek}, where based on the
Riemann--Liouville fractional derivatives the symmetric fractional
derivative was introduced. Klimek's approach is obtained choosing
our parameter $\gamma$ to be $1/2$.
Although the symmetric fractional derivative of Riemann--Liouville
introduced by Klimek is a useful tool in the description of some nonconservative
models, this type of differentiation does not seems suitable for all
kinds of variational problems. The hypothesis that admissible
trajectories $y$ have continuous symmetric fractional derivatives
implies that $y(a)=y(b)=0$ (\textrm{cf}. \cite{Ross}). Therefore, the
advantage of the fractional derivative ${{^CD}^{\alpha,\beta}_{\gamma}}$
via Caputo lies in the fact that using this derivative
we can describe a more general class of
variational problems. It is also worth pointing out that our
fractional derivative ${{^CD}^{\alpha,\beta}_{\gamma}}$
allows to generalize the results
presented in \cite{agrawalCap}.

The text is organized as follows. Section~\ref{sec2} presents some
preliminaries. In Section~\ref{main:1} we introduce the fractional
derivative ${{^CD}^{\alpha,\beta}_{\gamma}}$
and provide the necessary concepts and results
needed in the sequel. Our main results are stated and proved in
Section~\ref{ssec:pro}. The fractional variational problems under
our consideration are formulated in terms of the fractional
derivative ${{^CD}^{\alpha,\beta}_{\gamma}}$.
We discuss the fundamental concepts of a
variational calculus such as the Euler--Lagrange equations for the
basic (Subsection~\ref{ssec:EL}) and isoperimetric
(Subsection~\ref{sec:iso}) problems, as well as transversality
conditions (Subsection~\ref{sec:tran}). We end with conclusions in
Section~\ref{Con}.


\section{Preliminaries}
\label{sec2}

In this section we review the necessary definitions and facts from
fractional calculus. For more on the subject we refer the reader to
\cite{kilbas,Oldham,Podlubny,samko}.
Let $f\in L_1([a,b])$  and $0<\alpha<1$. We begin with the left and
the right Riemann--Liouville Fractional Integrals (RLFI) of order
$\alpha$ of a function $f$. The left RLFI is defined by
\begin{equation}
\label{RLFI1} {_aI_x^\alpha}f(x)=\frac{1}{\Gamma(\alpha)}\int_a^x
(x-t)^{\alpha-1}f(t)dt,\quad x\in[a,b],
\end{equation}
while the right RLFI is given by
\begin{equation}
\label{RLFI2}
{_xI_b^\alpha}f(x)=\frac{1}{\Gamma(\alpha)}\int_x^b(t-x)^{\alpha-1}
f(t)dt,\quad x\in[a,b],
\end{equation}
where $\Gamma(\cdot)$ represents the Gamma function. Moreover,
${_aI_x^0}f={_xI_b^0}f=f$ if $f$ is a continuous function.

Let $f\in AC([a,b])$, where $AC([a,b])$ represents the space of
absolutely continuous functions on $[a,b]$. Using equations
\eqref{RLFI1} and \eqref{RLFI2}, we define the left and the right
Riemann--Liouville and Caputo derivatives as follows. The left
Riemann--Liouville Fractional Derivative (RLFD) is given by
\begin{equation}
\label{RLFD1}
\begin{split}
{_aD_x^\alpha}f(x)&=\frac{1}{\Gamma(1-\alpha)}\frac{d}{dx}
\int_a^x (x-t)^{-\alpha}f(t)dt\\
&=\frac{d}{dx}{_aI_x^{1-\alpha}}f(x),\quad x\in[a,b],
\end{split}
\end{equation}
the right RLFD by
\begin{equation}
\label{RLFD2}
\begin{split}
{_xD_b^\alpha}f(x)
&=\frac{-1}{\Gamma(1-\alpha)}\frac{d}{dx}
\int_x^b (t-x)^{-\alpha} f(t)dt\\
&=\left(-\frac{d}{dx}\right){_xI_b^{1-\alpha}}f(x),
\quad x\in[a,b],
\end{split}
\end{equation}
the left Caputo Fractional Derivative (CFD) is defined by
\begin{equation}
\label{CFD1}
\begin{split}
{^C_aD_x^\alpha}f(x)
&=\frac{1}{\Gamma(1-\alpha)}
\int_a^x (x-t)^{-\alpha}\frac{d}{dt}f(t)dt\\
&={_aI_x^{1-\alpha}}\frac{d}{dx}f(x),
\quad x\in[a,b],
\end{split}
\end{equation}
and the right CFD by
\begin{equation}
\label{CFD2}
\begin{split}
{^C_xD_b^\alpha}f(x)
&=\frac{-1}{\Gamma(1-\alpha)}
\int_x^b (t-x)^{-\alpha} \frac{d}{dt}f(t)dt\\
&={_xI_b^{1-\alpha}}\left(-\frac{d}{dx}\right)f(x),
\quad x\in[a,b],
\end{split}
\end{equation}
where $\alpha$ is the order of the derivative.
The operators \eqref{RLFI1}--\eqref{CFD2} are obviously linear. We
now present the rule of fractional integration by parts for RLFI
(see for instance \cite{int:partsRef}). Let $0<\alpha<1$, $p\geq1$,
$q \geq 1$, and $1/p+1/q\leq1+\alpha$. If $g\in L_p([a,b])$ and
$f\in L_q([a,b])$, then
\begin{equation}
\label{ipi}
\int_{a}^{b} g(x){_aI_x^\alpha}f(x)dx
=\int_a^b f(x){_x I_b^\alpha} g(x)dx.
\end{equation}
In the discussion to follow, we will also need the following
formulae for fractional integrations by parts:
\begin{equation}
\label{ip}
\begin{gathered}
\int_{a}^{b} g(x) \, {^C_aD_x^\alpha}f(x)dx
=\left.f(x){_x I_b^{1-\alpha}} g(x)\right|^{x=b}_{x=a}
+\int_a^b f(x){_x D_b^\alpha} g(x)dx,\\
\int_{a}^{b} g(x) \, {^C_xD_b^\alpha}f(x)dx
=\left.-f(x){_a I_x^{1-\alpha}} g(x)\right|^{x=b}_{x=a}
+\int_a^b f(x){_a D_x^\alpha} g(x)dx.
\end{gathered}
\end{equation}

They can be easily derived using equations \eqref{RLFD1}--\eqref{CFD2},
the identity \eqref{ipi}, and performing integration by parts.


\section{The fractional ${{^CD}^{\alpha,\beta}_{\gamma}}$ operator}
\label{main:1}

Let  $\alpha, \beta \in(0,1)$ and $\gamma\in [0,1]$.
Motivated by the diamond-alpha operator
used in time scales (see, \textrm{e.g.},
\cite{MR2562284,MR2410768}),
we define the fractional derivative operator
$^CD^{\alpha,\beta}_{\gamma}$ by
\begin{equation}
\label{op}
{{^CD}^{\alpha,\beta}_{\gamma}} =\gamma \capla + (1-\gamma) \caprb \, ,
\end{equation}
which acts on $f\in AC([a,b])$ in the following way:
\begin{equation*}
{{^CD}^{\alpha,\beta}_{\gamma}} f(x)=\gamma \capla f(x) + (1-\gamma) \caprb f(x).
\end{equation*}
Note that
\begin{equation*}
^CD^{\alpha,\beta}_{0} f(x)=\caprb f(x),
\quad ^CD^{\alpha,\beta}_{1} f(x)=\capla f(x).
\end{equation*}

For $\mathbf{f}=[f_1,\ldots,f_N]:[a,b]\rightarrow \R^{N}$,
($N\in\N$) and $f_i \in AC([a,b])$, $i=1,\ldots,N$, we have
$$
{{^CD}^{\alpha,\beta}_{\gamma}} \mathbf{f}(x)
=[{{^CD}^{\alpha,\beta}_{\gamma}} f_1(x),
\ldots,{{^CD}^{\alpha,\beta}_{\gamma}} f_N(x)].
$$

The operator \eqref{op} is obviously linear. Using equations
\eqref{ip} it is easy to derive the following rule of fractional
integration by parts for ${{^CD}^{\alpha,\beta}_{\gamma}}$:
\begin{multline}
\label{byparts}
\int_{a}^{b}  g(x) \, {{^CD}^{\alpha,\beta}_{\gamma}} f(x)dx
=\gamma\left[f(x)\intra g(x)\right]^{x=b}_{x=a}\\
+(1-\gamma)\left[-f(x)\intlb g(x)\right]^{x=b}_{x=a}\\
+\int_a^b f(x)\opbac g(x)dx,
\end{multline}
where $\opbac = (1-\gamma) {_a}D_x^\beta + \gamma \, {_x}D_b^\alpha$,
which acts on $f$ as
$$
\opbac f(x) = (1-\gamma) {_a}D_x^\beta f(x) + \gamma \, {_x}D_b^\alpha f(x) \, .
$$

Let $\mathbf{D}$ denote the set of all functions
$\mathbf{y}:[a,b]\rightarrow \R^{N}$ such that
${{^CD}^{\alpha,\beta}_{\gamma}} \mathbf{y}$ exists and is continuous on the
interval $[a,b]$. We endow $\mathbf{D}$ with the following norm:
\begin{equation*}
    \|\mathbf{y}\|_{1,\infty}:=\max_{a\leq x \leq
    b}\|\mathbf{y}(x)\|+\max_{a\leq x \leq
    b}\|{{^CD}^{\alpha,\beta}_{\gamma}}\mathbf{y}(x)\|,
\end{equation*}
where $\|\cdot\|$ stands for a norm in $\R^N$.

Along the work we denote by $\partial_i K$, $i=1,\ldots,M$ ($M\in
\N$), the partial derivative of function $K:\R^M\rightarrow \R$
with respect to its $i$th argument.

Let $\lambda \in \R^r$. For simplicity of notation
we introduce the operators $[\cdot]$
and $\{\cdot\}_\lambda$ defined by
\begin{equation*}
\begin{split}
[\mathbf{y}](x) &= \left(x,\mathbf{y}(x),{{^CD}^{\alpha,\beta}_{\gamma}}
\mathbf{y}(x)\right) \, , \\
\{\mathbf{y}\}_\lambda(x) &= \left(x,\mathbf{y}(x),{{^CD}^{\alpha,\beta}_{\gamma}}
\mathbf{y}(x),\lambda_1,\ldots,\lambda_r\right) \, .
\end{split}
\end{equation*}


\section{Calculus of variations via ${{^CD}^{\alpha,\beta}_{\gamma}}$}
\label{ssec:pro}

We are concerned with the problem of finding minima (or maxima) of a functional
$\mathcal{J}: \mathcal{D}\rightarrow \R$, where $ \mathcal{D}$ is a
subset of $\mathbf{D}$. The formulation of a problem of
calculus of variations requires two steps: the specification of a
performance criterion and the statement of physical constraints that
should be satisfied.

A performance criterion $\mathcal{J}$, also called cost functional
(or cost), must be specified for evaluating the performance of a
system quantitatively. We consider the following cost:
\begin{equation*}
\mathcal{J}(\mathbf{y})=\int_a^b L[\mathbf{y}](x) \, dx,
\end{equation*}
where $x\in [a,b]$ is the independent variable, often called time;
$\mathbf{y}(x)\in \R^N$ is a real vector variable, the functions
$\mathbf{y}$ are generally called trajectories or curves;
${{^CD}^{\alpha,\beta}_{\gamma}}
\mathbf{y}(x)\in \R^N$ stands for the fractional derivative of
$\mathbf{y}(x)$; and $L\in C^1([a,b]\times\mathbb{R}^{2N};
\mathbb{R})$ is called a Lagrangian.

Enforcing constraints in the optimization problem reduces the set of
candidate functions and leads to the following definition.

\begin{definition}
A trajectory $\mathbf{y}\in \mathbf{D}$ is said to be an admissible
trajectory provided that it satisfies all of the constraints along
interval $[a,b]$. The set of admissible trajectories is defined as
\begin{equation*}
\mathcal{D}:=\{\mathbf{y}\in \mathbf{D}:\mathbf{y}
\mbox{ is admissible}\}.
\end{equation*}
\end{definition}

We now define what is meant by a minimizer
of $\mathcal{J}$ on $\mathcal{D}$.

\begin{definition}
A trajectory $\bar{\mathbf{y}}\in \mathcal{D}$ is said to be a local
minimizer for $\mathcal{J}$ on $\mathcal{D}$ if there exists
$\delta>0$ such that $\mathcal{J}(\bar{\mathbf{y}})\leq
\mathcal{J}(\mathbf{y})$ for all $\mathbf{y}\in \mathcal{D}$ with
$\|\mathbf{y}-\bar{\mathbf{y}}\|_{1,\infty}<\delta$.
\end{definition}

The concept of variation of a functional
is central to the solution of problems
of the calculus of variations.

\begin{definition}
The first variation of $\mathcal{J}$ at $\mathbf{y}\in \mathbf{D}$
in the direction $\mathbf{h}\in \mathbf{D}$ is defined as
\begin{equation*}
\delta\mathcal{J}(\mathbf{y};\mathbf{h})
:=\lim_{\varepsilon\rightarrow 0}\frac{\mathcal{J}(\mathbf{y}
+\varepsilon\mathbf{h})-\mathcal{J}(\mathbf{y})}{\varepsilon}
=\left.\frac{\partial}{\partial\varepsilon}\mathcal{J}(\mathbf{y}
+\varepsilon\mathbf{h})\right|_{\varepsilon=0}
\end{equation*}
provided the limit exists.
\end{definition}

\begin{definition}
A direction $\mathbf{h}\in \mathbf{D}$, $\mathbf{h}\neq 0$, is said
to be an admissible variation at $\mathbf{y}\in \mathcal{D}$ for
$\mathcal{J}$ if
\begin{itemize}
\item[(i)] $\delta\mathcal{J}(\mathbf{y};\mathbf{h})$ exists; and
\item[(ii)] $\mathbf{y}+\varepsilon\mathbf{h}\in \mathcal{D}$ for
all sufficiently small $\varepsilon$.
\end{itemize}
\end{definition}

The following well known result (see, \textrm{e.g.},
\cite[Proposition~5.5]{Trout}) offers a necessary
optimality condition for problems of calculus of variations
based on the concept of variations.

\begin{theorem}[\cite{Trout}]
\label{nesse_con}
Let $\mathcal{J}$ be a functional defined on
$\mathcal{D}$. Suppose that $\mathbf{y}$ is a local minimizer for
$\mathcal{J}$ on $\mathcal{D}$. Then,
$\delta\mathcal{J}(\mathbf{y};\mathbf{h})=0$ for each admissible
variation $\mathbf{h}$ in $\mathbf{y}$.
\end{theorem}


\subsection{Elementary problem of the ${{^CD}^{\alpha,\beta}_{\gamma}}$ calculus of variations}
\label{ssec:EL}

Let us begin with the following problem:
\begin{equation}
\label{Funct1}
\mathcal{J}(\mathbf{y})
=\int_a^b L[\mathbf{y}](x) \, dx \longrightarrow \min\\
\end{equation}
over all $\mathbf{y}\in \mathbf{D}$ satisfying the boundary conditions
\begin{equation}
\label{boun2}
\mathbf{y}(a)=\mathbf{y}^{a},
\quad \mathbf{y}(b)=\mathbf{y}^{b},
\quad \mathbf{y}^{a},\mathbf{y}^{b}\in \R^N.
\end{equation}

Next theorem gives the fractional Euler--Lagrange equation for the
problem \eqref{Funct1}--\eqref{boun2}.

\begin{theorem}
\label{Theo E-L1}
Let $\mathbf{y}=(y_1,\ldots,y_N)$ be a local
minimizer to problem \eqref{Funct1}--\eqref{boun2}. Then,
$\mathbf{y}$ satisfies the system of
$N$ fractional Euler--Lagrange equations
\begin{equation}
\label{E-L1}
\partial_iL[\mathbf{y}](x)+\opbac
\partial_{N+i}L[\mathbf{y}](x)=0, \quad
i=2,\ldots N+1,
\end{equation}
for all $x\in[a,b]$.
\end{theorem}

\begin{proof}
Suppose that $\mathbf{y}$ is a local minimizer for $\mathcal{J}$.
Let $\mathbf{h}$ be an arbitrary admissible variation for problem
\eqref{Funct1}--\eqref{boun2}, \textrm{i.e.}, $h_i(a)=h_i(b)=0$,
$i=1,\ldots,N$. Based on the differentiability properties of $L$ and
Theorem~\ref{nesse_con}, a necessary condition for $\mathbf{y}$ to be
a local minimizer is given by
$$
\left.\frac{\partial}{\partial\varepsilon}\mathcal{J}(\mathbf{y}
+\varepsilon\mathbf{h})\right|_{\varepsilon=0} = 0 \, ,
$$
that is,
\begin{multline}
\label{eq:FT}
\int_a^b\Biggl[\sum_{i=2}^{N+1}\partial_iL[\mathbf{y}](x)h_{i-1}(x)\\
+\sum_{i=2}^{N+1}\partial_{N+i}L[\mathbf{y}](x){{^CD}^{\alpha,\beta}_{\gamma}}
h_{i-1}(x)\Biggr]dx=0.
\end{multline}
Using formula \eqref{byparts} for integration by parts in the
second term of the integrand function, we get
\begin{multline}
\label{eq:aft:IP}
\int_a^b\left[\sum_{i=2}^{N+1}\partial_i
L[\mathbf{y}](x)+\opbac\partial_{N+i}L[\mathbf{y}](x)\right]h_{i-1}(x)dx\\
+\gamma\left.\left[\sum_{i=2}^{N+1}h_{i-1}(x)\intra
\partial_{N+i}L[\mathbf{y}](x)\right]\right|^{x=b}_{x=a}\\
-(1-\gamma)\left.\left[\sum_{i=2}^{N+1}h_{i-1}(x)\intlb
\partial_{N+i}L[\mathbf{y}](x)\right]\right|^{x=b}_{x=a}=0.
\end{multline}
Since $h_i(a)=h_i(b)=0$, $i=1,\ldots,N$, by the fundamental lemma of
the calculus of variations we deduce that
\begin{equation*}
\partial_iL[\mathbf{y}](x)+\opbac\partial_{N+i}L[\mathbf{y}](x)=0,
\quad i=2,\ldots,N+1,
\end{equation*}
for all $x\in[a,b]$.
\end{proof}

Observe that if $\alpha$ and $\beta$ go to $1$, then $\capla$ and
${_a}D_x^\alpha$ can be replaced with $\frac{d}{dx}$;
and $\caprb$ and ${_x}D_b^\beta$  with $-\frac{d}{dx}$
(see, \textrm{e.g.}, \cite{Podlubny}).
Thus, if $\gamma=1$ or $\gamma=0$,
then for $\alpha,\beta \rightarrow 1$ we obtain
a corresponding result in the classical context
of the calculus of variations (see, \textrm{e.g.}, \cite{Trout}).


\subsection{${{^CD}^{\alpha,\beta}_{\gamma}}$ transversality conditions}
\label{sec:tran}

Let $l\in \{1,\ldots,N\}$. Assume that $\mathbf{y}(a)=\mathbf{y}^a$,
$y_i(b)=y_i^b$, $i=1,\ldots,N$, $i\neq l$, but $y_l(b)$ is free.
Then, $h_l(b)$ is free and by equations \eqref{E-L1} and
\eqref{eq:aft:IP} we obtain
\begin{equation}
\label{new:bcb}
\Bigl[\gamma \intra \partial_{l+1}L[\mathbf{y}](x)
\left.-(1-\gamma) \intlb \partial_{N+1+l}L[\mathbf{y}](x)\Bigr]\right|_{x=b}=0.
\end{equation}

Let us consider now the case when $\mathbf{y}(a)=\mathbf{y}^a$,
$y_i(b)=y_i^b$, $i=1,\ldots,N$, $i\neq l$, and  $y_l(b)$ is free
but restricted by a terminal condition $y_l(b)\leq y^{b}_l$. Then, in
the optimal solution $\mathbf{y}$ we have two possible types of
outcome: $y_l(b)< y^{b}_l$ or $y_l(b)= y^{b}_l$. If $y(b)<
y^{b}_l$, then there are admissible neighboring paths with terminal
value both above and below $y_l(b)$, so that $h_l(b)$ can take
either sign. Therefore, the transversality conditions is
\begin{equation}
\label{tran:1}
\Bigl[\gamma \intra \partial_{l+1}L[\mathbf{y}](x)
\left.-(1-\gamma) \intlb \partial_{n+1+l}L[\mathbf{y}](x)\Bigr]\right|_{x=b}=0
\end{equation}
for $y_l(b)< y_l^{b}$. The other outcome $y_l(b)= y^{b}_l$
only admits the neighboring paths with terminal value
$\tilde{y}_l(b) \leq y_l(b)$. Assuming, without loss of generality,
that $h_l(b)\geq 0$, this means that $\varepsilon \leq 0$.
Hence, the transversality condition, which has
it root in the first order condition \eqref{eq:FT}, must be changed
to the inequality. For a minimization problem, the $\leq$ type of
inequality is called for, and we obtain
\begin{equation}
\label{tran:2}
\Bigl[\gamma \intra \partial_{l+1} L[\mathbf{y}](x)
\left.-(1-\gamma) \intlb \partial_{N+1+l}
L[\mathbf{y}](x)\Bigr]\right|_{x=b}\leq0
\end{equation}
for $y_l(b)= y^{b}_l$. Combining \eqref{tran:1}
and \eqref{tran:2}, we may write
the following transversality
condition for a minimization problem:
\begin{gather*}
\label{tran}
\Bigl[\gamma \intra \partial_{l+1}L[\mathbf{y}](x)
\left.-(1-\gamma) \intlb \partial_{N+1+l}
L[\mathbf{y}](x)\Bigr]\right|_{x=b}\leq0,\\
y_l(b)\leq y^{b}_l,\\
(y_l(b)-y^{b}_l)\Bigl[\gamma \intra\partial_{l+1}L[\mathbf{y}](x)
\left.-(1-\gamma) \intlb \partial_{N+1+l}
L[\mathbf{y}](x)\Bigr]\right|_{x=b}=0.
\end{gather*}


\subsection{The ${{^CD}^{\alpha,\beta}_{\gamma}}$ isoperimetric problem}
\label{sec:iso}

Let us consider now the isoperimetric problem that consists of
minimizing \eqref{Funct1} over all $\mathbf{y}\in \mathbf{D}$
satisfying $r$ isoperimetric constraints
\begin{equation}
\label{cons2:iso}
\mathcal{G}^j(\mathbf{y})=\int_{a}^{b} G^j[\mathbf{y}](x)dx=l_j,
\quad j=1,\ldots,r,
\end{equation}
where $G^j\in C^1([a,b]\times\mathbb{R}^{2N}; \mathbb{R})$,
$j=1,\ldots,r$, and boundary conditions \eqref{boun2}.
Necessary optimality conditions for isoperimetric problems can
be obtained by the following general theorem (see, \textrm{e.g.},
\cite[Theorem~5.16]{Trout}).

\begin{theorem}
\label{nes:iso}
Let $\mathcal{J},\mathcal{G}^{1},\ldots,\mathcal{G}^r$ be
functionals defined in a neighborhood of $\mathbf{y}$ and having
continuous first variations in this neighborhood. Suppose that
$\mathbf{y}$ is a local minimizer of \eqref{Funct1} subject to the
boundary conditions \eqref{boun2} and the isoperimetric constrains
\eqref{cons2:iso}.
Then, either:\\
(i) for all $\mathbf{h}^{j}\in \mathbf{D}$, $j=1,\ldots,r$,
\begin{equation}
\label{lmi}
\left|
  \begin{array}{cccc}
    \delta\mathcal{G}^{1}(\mathbf{y};\mathbf{h}^{1}) &
    \delta\mathcal{G}^{1}(\mathbf{y};\mathbf{h}^{2})&
    \cdots & \delta\mathcal{G}^{1}(\mathbf{y};\mathbf{h}^{r})\\
    \delta\mathcal{G}^{2}(\mathbf{y};\mathbf{h}^{1})
    & \delta\mathcal{G}^{2}(\mathbf{y};\mathbf{h}^{2})
    & \cdots & \delta\mathcal{G}^{2}(\mathbf{y};\mathbf{h}^{r}) \\
    \vdots & \vdots& \ddots & \vdots \\
    \delta\mathcal{G}^{r}(\mathbf{y};\mathbf{h}^{1})
    & \delta\mathcal{G}^{r}(\mathbf{y};\mathbf{h}^{2})
    & \cdots & \delta\mathcal{G}^{r}(\mathbf{y};\mathbf{h}^{r})\\
  \end{array}
\right|=0
\end{equation}
or\\
(ii) there exist constants $\lambda_{j}\in \R$,
$j=1,\ldots,r$, for which
\begin{equation*}
\delta\mathcal{J}(\mathbf{y};\mathbf{h})
=\sum_{j=1}^{m}\lambda_{j}\delta\mathcal{G}^{j}(\mathbf{y};\mathbf{h})
\quad \forall \mathbf{h} \in \mathbf{D}.
\end{equation*}
\end{theorem}

Note that condition (ii) of Theorem~\ref{nes:iso}
can be written in the form
\begin{equation}
\label{lmiii}
\delta \left(\mathcal{J}(\mathbf{y};\mathbf{h})
-\sum_{j=1}^{r}\lambda_{j}\mathcal{G}^{j}(\mathbf{y};\mathbf{h})\right)=0.
\end{equation}

Suppose now that assumptions of Theorem~\ref{nes:iso} hold but
condition (i) does not hold. Then, equation \eqref{lmiii} is
fulfilled for every $\mathbf{h} \in \mathbf{D}$. Let us consider
function $\mathbf{h}$ such that $\mathbf{h}(a)=\mathbf{h}(b)=0$
and denote by $\mathcal{F}$ the functional
$$
\mathcal{J}-\sum_{j=1}^{r}\lambda _{j}\mathcal{G}^{j}.
$$
Then, we have
\begin{equation*}
\begin{split}
0 &= \delta \mathcal{F}(\mathbf{y};\mathbf{h})
=\frac{\partial}{\partial \varepsilon}\mathcal{F}(\mathbf{y}
+\varepsilon\mathbf{h})|_{\varepsilon=0}\\
&=\int_a^b\Biggl[\sum_{i=2}^{N+1}\partial_iF\{\mathbf{y}\}_\lambda(x)h_{i-1}(x)
+\sum_{i=2}^{N+1}\partial_{N+i}F\{\mathbf{y}\}_\lambda(x){{^CD}^{\alpha,\beta}_{\gamma}}
h_{i-1}(x)\Biggr]dx,
\end{split}
\end{equation*}
where the function $F:[a,b]\times \R^{2N}\times \R^r \rightarrow \R$
is defined by
$$
F\{\mathbf{y}\}_\lambda(x)=L[\mathbf{y}](x)
-\sum_{j=1}^{r}\lambda_{j} G^{j}[\mathbf{y}](x).
$$
On account of the above, and similarly in spirit
to the proof of Theorem~\ref{Theo E-L1}, we obtain
\begin{equation}
\label{ele}
\partial_iF\{\mathbf{y}\}_\lambda(x)+\opbac
\partial_{N+i}F\{\mathbf{y}\}_\lambda(x)=0,
\quad i=2,\ldots N+1.
\end{equation}
Therefore, we have the following necessary optimality condition for
the isoperimetric problem:

\begin{theorem}
\label{Th:B:EL-CV}
Let assumptions of Theorem~\ref{nes:iso} hold.
If $\mathbf{y}$ is a local minimizer to the
isoperimetric problem \eqref{Funct1},\eqref{boun2} and
\eqref{cons2:iso}, and condition \eqref{lmi} does not hold,
then $\mathbf{y}$ satisfies the system of $N$ fractional
Euler--Lagrange equations \eqref{ele} for all $x\in[a,b]$.
\end{theorem}

Suppose now that constraints \eqref{cons2:iso}
are characterized by inequalities
\begin{equation*}
\mathcal{G}^j(\mathbf{y})
=\int_{a}^{b} G^j[\mathbf{y}](x)dx\leq l_j,
\quad j=1,\ldots,r.
\end{equation*}
In this case we can set
\begin{equation*}
\int_{a}^{b}\left(G^j[\mathbf{y}](x)-\frac{l_j}{b-a}\right)dx
+\int_{a}^{b}(\phi_j(x))^2dx=0,
\end{equation*}
$j=1,\ldots,r$, where $\phi_j$ have the same continuity
properties as $y_i$. Therefore, we obtain the following problem:
\begin{equation*}
\hat{\mathcal{J}}(y)=\int_a^b \hat{L}(x,\mathbf{y}(x),{{^CD}^{\alpha,\beta}_{\gamma}}
\mathbf{y}(x), \mathbf{\phi}(x)) \, dx \longrightarrow \min
\end{equation*}
where $\mathbf{\phi}(x)=[\phi_1(x),\ldots,\phi_r(x)]$,
subject to $r$ isoperimetric constraints
\begin{equation*}
\int_{a}^{b}\left[G^j[\mathbf{y}](x)
-\frac{l_j}{b-a}+(\phi_j(x))^2\right]dx=0,
\quad j=1,\ldots,r,
\end{equation*}
and boundary conditions \eqref{boun2}. Assuming that assumptions of
Theorem~\ref{Th:B:EL-CV} are satisfied, we conclude that there exist
constants $\lambda_{j}\in \R$, $j=1,\ldots,r$, for which the
system of equations
\begin{multline}
\label{iso:1:L:EL}
\opbac \partial_{N+i}F(x,\mathbf{y}(x),{{^CD}^{\alpha,\beta}_{\gamma}}
\mathbf{y}(x),\lambda_1,
\ldots,\lambda_r,\mathbf{\phi}(x))\\
+\partial_iF(x,\mathbf{y}(x),{{^CD}^{\alpha,\beta}_{\gamma}}
\mathbf{y}(x),\lambda_1,\ldots,\lambda_r,\mathbf{\phi}(x))=0,
\end{multline}
$i=2,\ldots, N+1$, where
$F=\hat{L}+\sum_{j=1}^r\lambda_j(G^j-\frac{l_j}{b-a}+\phi_j^2)$ and
\begin{equation}
\label{iso:2:L:EL}
\lambda_j\phi_j(x)=0,  \quad j=1,\ldots,r,
\end{equation}
hold for all $x\in[a,b]$. Note that it is enough to assume that the
regularity condition holds for the constraints which are active at
the local minimizer $\mathbf{y}$. Indeed, suppose that $l<r$ constrains,
say $\mathcal{G}^1,\ldots,\mathcal{G}^l$ for simplicity, are active
at the local minimizer $\mathbf{y}$, and
\begin{equation*}
\left|
  \begin{array}{cccc}
    \delta\mathcal{G}^{1}(\mathbf{y};\mathbf{h}^{1})
    & \delta\mathcal{G}^{1}(\mathbf{y};\mathbf{h}^{2})
    & \cdots & \delta\mathcal{G}^{1}(\mathbf{y};\mathbf{h}^{l})\\
    \delta\mathcal{G}^{2}(\mathbf{y};\mathbf{h}^{1})
    & \delta\mathcal{G}^{2}(\mathbf{y};\mathbf{h}^{2})
    & \cdots & \delta\mathcal{G}^{2}(\mathbf{y};\mathbf{h}^{l}) \\
    \vdots & \vdots& \ddots & \vdots \\
    \delta\mathcal{G}^{l}(\mathbf{y};\mathbf{h}^{1})
    & \delta\mathcal{G}^{l}(\mathbf{y};\mathbf{h}^{2})
    & \cdots & \delta\mathcal{G}^{l}(\mathbf{y};\mathbf{h}^{l})
  \end{array}
\right|\neq 0
\end{equation*}
for (independent) $\mathbf{h}^{j}\in \mathbf{D}$, $j=1,\ldots,l$.
Since the inequality constraints
$\mathcal{G}^{l+1},\ldots,\mathcal{G}^r$ are inactive, the
conditions \eqref{iso:2:L:EL} are trivially satisfied by taking
$\lambda_{l+1}=\cdots=\lambda_{r}=0$. On the other hand, since the
inequality constraints $\mathcal{G}^1,\ldots,\mathcal{G}^l$ are
active and satisfy a regularity condition at $\mathbf{y}$, the
conclusion that there exist constants $\lambda_{j}\in \R$,
$j=1,\ldots,r$, such that \eqref{iso:1:L:EL} holds follow from
Theorem~\ref{Th:B:EL-CV}. Moreover, \eqref{iso:2:L:EL} is trivially
satisfied for $j=1,\ldots,l$.


\section{Conclusion}
\label{Con}

The FC is a mathematical area of a currently strong research, with
numerous applications in physics and engineering. The fractional
operators are non-local, therefore they are suitable for
constructing models possessing memory effect. This gives several
possible applications of the FCV in describing non-local properties
of physical systems in mechanics or electrodynamics. In this note we
extend the notions of Caputo fractional derivative to the
fractional derivative ${{^CD}^{\alpha,\beta}_{\gamma}}$.
We emphasize that this derivative
allows to describe a more general class of variational problems
and, as a particular case, we get the results of \cite{agrawalCap}.

Knowing the importance and relevance of multiobjective problems of the
calculus of variations in physics and engineering, our further
research will continue towards multiobjective FCV. This is a
completely open research area and will be addressed elsewhere.


\section*{Acknowledgements}

This work was partially presented at the
\emph{IFAC Workshop on Fractional Derivatives and Applications} (IFAC FDA'2010),
held in University of Extremadura, Badajoz, Spain, October 18-20, 2010.
It was supported by {\it FEDER} funds through
{\it COMPETE} --- Operational Programme Factors of Competitiveness
(``Programa Operacional Factores de Competitividade'')
and by Portuguese funds through the
{\it Center for Research and Development
in Mathematics and Applications} (University of Aveiro)
and the Portuguese Foundation for Science and Technology
(``FCT --- Funda\c{c}\~{a}o para a Ci\^{e}ncia e a Tecnologia''),
within project PEst-C/MAT/UI4106/2011
with COMPETE number FCOMP-01-0124-FEDER-022690.
The first author was also supported by Bia{\l}ystok
University of Technology grant S/WI/02/2011.



\bigskip \smallskip

\it

\noindent
$^1$ Faculty of Computer Science \\
Bia{\l}ystok University of Technology \\
15-351 Bia\l ystok, POLAND  \\[4pt]
e-mail: abmalinowska@ua.pt \\[12pt]
$^2$ Center for Research and Development in Mathematics and Applications \\
Department of Mathematics, University of Aveiro \\
3810-193 Aveiro, PORTUGAL \\[4pt]
e-mail: delfim@ua.pt



\begin{thebibliography}{99}

\normalsize

\bibitem{Agrawal:2002}
O. P. Agrawal,
Formulation of Euler-Lagrange equations for fractional variational problems.
\emph{J. Math. Anal. Appl.} {\bf 272}, No~1 (2002), 368--379.

\bibitem{agrawalCap}
O. P. Agrawal,
Generalized Euler-Lagrange equations and transversality
conditions for FVPs in terms of the Caputo derivative.
\emph{J. Vib. Control} {\bf 13}, No~9-10 (2007), 1217--1237.

\bibitem{CD:Agrawal:2007}
O. P. Agrawal,
Fractional variational calculus in terms
of Riesz fractional derivatives.
\emph{J. Phys. A} {\bf 40}, No~24 (2007), 6287--6303.

\bibitem{R:A:D:10}
R. Almeida, A. B. Malinowska, D. F. M. Torres,
A fractional calculus of variations for multiple integrals
with application to vibrating string.
\emph{J. Math. Phys.} {\bf 51}, No~3 (2010), 033503, 12~pp.
{\tt arXiv:1001.2722}

\bibitem{Almeida2}
R. Almeida, D. F. M. Torres,
Calculus of variations with fractional derivatives and fractional integrals.
\emph{Appl. Math. Lett.} {\bf 22}, No~12 (2009), 1816--1820.
{\tt arXiv:0907.1024}

\bibitem{comRic:Leitmann}
R. Almeida, D. F. M. Torres,
Leitmann's direct method for fractional optimization problems.
\emph{Appl. Math. Comput.} {\bf 217}, No~3 (2010), 956--962.
{\tt arXiv:1003.3088}

\bibitem{comRic:RL-modified}
R. Almeida, D. F. M. Torres,
Fractional variational calculus for nondifferentiable functions.
\emph{Comput. Math. Appl.} {\bf 61}, No~10 (2011), 3097--3104.
{\tt arXiv:1103.5406}

\bibitem{Atanackovic}
T. M. Atanackovi\'c, S. Konjik, S. Pilipovi\'c,
Variational problems with fractional derivatives: Euler-Lagrange equations.
\emph{J. Phys. A} {\bf 41}, No~9 (2008), 095201, 12~pp.
{\tt arXiv:1101.2961}

\bibitem{ReviewerAsked02}
D. Baleanu,
Fractional variational principles in action.
\emph{Phys. Scripta} {\bf T136} (2009), Article Number: 014006.

\bibitem{Baleanu:Agrawal}
D. Baleanu, O. P. Agrawal,
Fractional Hamilton formalism within Caputo's derivative.
\emph{Czechoslovak J. Phys.} {\bf 56}, No~10-11 (2006), 1087--1092.

\bibitem{Baleanu4}
D. Baleanu, A. K. Golmankhaneh, R. Nigmatullin, A. K. Golmankhaneh,
Fractional Newtonian mechanics.
\emph{Cent. Eur. J. Phys.} {\bf 8}, No~1 (2010), 120--125.

\bibitem{Baleanu:Mus}
D. Baleanu, S. I. Muslih,
Lagrangian formulation of classical fields
within Riemann-Liouville fractional derivatives.
\emph{Phys. Scripta} {\bf 72}, No~2-3 (2005), 119--121.
{\tt arXiv:hep-th/0510071}

\bibitem{MyID:179}
N. R. O. Bastos, R. A. C. Ferreira, D. F. M. Torres,
Discrete-time fractional variational problems.
\emph{Signal Process.} {\bf 91}, No~3 (2011), 513--524.
{\tt arXiv:1005.0252}

\bibitem{int:partsRef}
R. Brunetti, D. Guido, R. Longo,
Modular structure and duality in conformal quantum field theory.
\emph{Comm. Math. Phys.} {\bf 156}, No~1 (1993), 201--219.

\bibitem{Cresson}
J. Cresson,
Fractional embedding of differential operators and Lagrangian systems.
\emph{J. Math. Phys.} {\bf 48}, No~3 (2007), 033504, 34~pp.
{\tt arXiv:math/0605752}

\bibitem{El-Nabulsi:Torres07}
R. A. El-Nabulsi, D. F. M. Torres,
Necessary optimality conditions for fractional
action-like integrals of variational calculus
with Riemann-Liouville derivatives of order $(\alpha,\beta)$.
\emph{Math. Methods Appl. Sci.} {\bf 30}, No~15 (2007), 1931--1939.
{\tt arXiv:math-ph/0702099}

\bibitem{El-Nabulsi:Torres:2008}
R. A. El-Nabulsi, D. F. M. Torres,
Fractional actionlike variational problems.
\emph{J. Math. Phys.} {\bf 49}, No~5 (2008), 053521, 7~pp.
{\tt arXiv:0804.4500}

\bibitem{Frederico:Torres1}
G. S. F. Frederico, D. F. M. Torres,
A formulation of Noether's theorem for fractional
problems of the calculus of variations.
\emph{J. Math. Anal. Appl.} {\bf 334}, No~2 (2007), 834--846.
{\tt arXiv:math/0701187}

\bibitem{NonlinearDyn}
G. S. F. Frederico, D. F. M. Torres,
Fractional conservation laws in optimal control theory.
\emph{Nonlinear Dynam.} {\bf 53}, No~3 (2008), 215--222.
{\tt arXiv:0711.0609}

\bibitem{withGasta:SI:Leit:85}
G. S. F. Frederico, D. F. M. Torres,
Fractional Noether's theorem in the Riesz-Caputo sense.
\emph{Appl. Math. Comput.} {\bf 217}, No~3 (2010), 1023--1033.
{\tt arXiv:1001.4507}

\bibitem{hilfer2}
R. Hilfer,
{\it Applications of fractional calculus in physics}.
World Sci. Publishing, River Edge, NJ (2000).

\bibitem{Jumarie4}
G. Jumarie,
Fractional Hamilton-Jacobi equation for the optimal control
of nonrandom fractional dynamics with fractional cost function.
\emph{J. Appl. Math. Comput.} {\bf 23}, No~1-2 (2007), 215--228.

\bibitem{Jumarie3b}
G. Jumarie,
An approach via fractional analysis to non-linearity induced
by coarse-graining in space.
\emph{Nonlinear Anal. Real World Appl.} {\bf 11}, No~1 (2010), 535--546.

\bibitem{kilbas}
A. A. Kilbas, H. M. Srivastava, J. J. Trujillo,
{\it Theory and applications of fractional differential equations}.
Elsevier, Amsterdam (2006).

\bibitem{klimek}
M. Klimek,
Stationarity-conservation laws for fractional
differential equations with variable coefficients.
\emph{J. Phys. A} {\bf 35}, No~31 (2002), 6675--6693.

\bibitem{MR2562284}
A. B. Malinowska, D. F. M. Torres,
On the diamond-alpha Riemann integral
and mean value theorems on time scales.
\emph{Dynam. Systems Appl.} {\bf 18}, No~3-4 (2009), 469--481.
{\tt arXiv:0804.4420}

\bibitem{MalTor2010}
A. B. Malinowska, D. F. M. Torres,
Generalized natural boundary conditions for fractional variational
problems in terms of the Caputo derivative.
\emph{Comput. Math. Appl.} {\bf 59}, No~9 (2010), 3110--3116.
{\tt arXiv:1002.3790}

\bibitem{MalTor}
A. B. Malinowska, D. F. M. Torres,
Natural boundary conditions in the calculus of variations.
\emph{Math. Methods Appl. Sci.} {\bf 33}, No~14 (2010), 1712--1722.
{\tt arXiv:0812.0705}

\bibitem{miller}
K. S. Miller, B. Ross,
{\it An introduction to the fractional
calculus and fractional differential equations}.
Wiley, New York (1993).

\bibitem{D:D}
D. Mozyrska, D. F. M. Torres,
Minimal modified energy control for fractional
linear control systems with the Caputo derivative.
\emph{Carpathian J. Math.} {\bf 26}, No~2 (2010), 210--221.
{\tt arXiv:1004.3113}

\bibitem{Odz:Tor}
T. Odzijewicz, D. F. M. Torres,
Fractional calculus of variations for double integrals.
\emph{Balkan J. Geom. Appl.} {\bf 16}, No~2 (2011), 102--113.
{\tt arXiv:1102.1337}

\bibitem{Oldham}
K. B. Oldham, J. Spanier,
{\it The fractional calculus}.
Academic Press [A subsidiary of Harcourt Brace Jovanovich,
Publishers], New York (1974).

\bibitem{MR2164553}
A. Yu. Plakhov, D. F. M. Torres,
Newton's aerodynamic problem in media of chaotically moving particles.
\emph{Mat. Sb.} {\bf 196}, No~6 (2005), 111--160;
translation in \emph{Sb. Math.} {\bf 196}, No~5-6 (2005), 885--933.
{\tt arXiv:math/0407406}

\bibitem{Podlubny}
I. Podlubny,
{\it Fractional differential equations}.
Academic Press, San Diego, CA (1999).

\bibitem{Rabei2}
E. M. Rabei, B. S. Ababneh,
Hamilton-Jacobi fractional mechanics.
\emph{J. Math. Anal. Appl.} {\bf 344}, No~2 (2008), 799--805.

\bibitem{ReviewerAsked01}
E. M. Rabei, K. I. Nawafleh, R. S. Hijjawi, S. I. Muslih, D. Baleanu,
The Hamilton formalism with fractional derivatives.
\emph{J. Math. Anal. Appl.} {\bf 327}, No~2 (2007), 891--897.

\bibitem{rie}
F. Riewe,
Nonconservative Lagrangian and Hamiltonian mechanics.
\emph{Phys. Rev. E} (3) {\bf 53}, No~2 (1996), 1890--1899.

\bibitem{Ross}
B. Ross, S. G. Samko, E. R. Love,
Functions that have no first order derivative might have
fractional derivatives of all orders less than one.
\emph{Real Anal. Exchange} {\bf 20}, No~1 (1994/95), 140--157.

\bibitem{samko}
S. G. Samko, A. A. Kilbas, O. I. Marichev,
{\it Fractional integrals and derivatives}.
Translated from the 1987 Russian original,
Gordon and Breach, Yverdon (1993).

\bibitem{MR2410768}
M. R. Sidi Ammi, R. A. C. Ferreira, D. F. M. Torres,
Diamond-$\alpha$ Jensen's inequality on time scales.
\emph{J. Inequal. Appl.} {\bf 2008}, Art. ID 576876 (2008), 13~pp.
{\tt arXiv:0712.1680}

\bibitem{Tarasov}
V. E. Tarasov,
Fractional vector calculus and fractional Maxwell's equations.
\emph{Ann. Physics} {\bf 323}, No~11 (2008), 2756--2778.
{\tt arXiv:0907.2363}

\bibitem{Trout}
J. L. Troutman,
{\it Variational calculus and optimal control}.
Second edition, Springer, New York (1996).

\bibitem{vanBrunt}
B. van Brunt,
{\it The calculus of variations}.
Springer, New York (2004).

\end{thebibliography}
\end{document}